\documentclass{Prueba}

\citesort

\theoremstyle{cupthm}
\newtheorem{Theorem}{Theorem}[section]
\newtheorem{Proposition}[Theorem]{Proposition}

\newtheorem{Lemma}[Theorem]{Lemma}
\theoremstyle{cupdefn}
\newtheorem{Definition}[Theorem]{Definition}
\theoremstyle{cuprem}
\newtheorem{Remark}[Theorem]{Remark}
\numberwithin{equation}{section}

\begin{document}
\runningtitle{An extension of the Cartwright-McMullen theorem in FC for the smooth Stieltjes case}
\title{An extension of the Cartwright-McMullen theorem in fractional calculus for the smooth Stieltjes case}
\author[1]{Daniel Cao Labora}
\address[1]{Departamento de An\'alise Matem\'atica, Estat\'istica e Optimizaci\'on,
Facultade de Matem\'aticas, Universidade de Santiago de Compostela; R\'ua Lope G\'omez de Marzoa, 6, ZIP: 15782, Spain;\email{daniel.cao@usc.gal} \\
ORCID: \url{https://orcid.org/0000-0003-2266-2075}}

\authorheadline{D. Cao Labora}

\dedication{Dedicated to the former work of Donald Cartwright and John R. McMullen, from University of Sydney, that settled a fundamental result in the discipline of fractional calculus in 1978.}

\support{The work of the author has been partially supported by the Agencia Estatal de Investigación (AEI) of Spain under Grant PID2020-113275GB-I00, cofinanced by the European Community fund FEDER, as well as Xunta de Galicia grant ED431C
2019/02 for Competitive Reference Research Groups (2019–22).}

\begin{abstract}
In 1976, Donald Cartwright and John McMullen characterized axiomatically the Riemann-Liouvile fractional integral in their paper \cite{CarMc}, that was published in 1978. The motivation for their work was to answer affirmatively to a conjecture stated by J. S. Lew a few years before, in 1972. Essentially, their ``Cartwright-McMullen theorem in fractional calculus'' proved that the Riemann-Liouville fractional integral is the only continuous extension of the usual integral operator to positive real orders, in such a way that the Index Law holds. In this paper, we propose an analogous result for the uniqueness of the extension of the Stieltjes integral operator, in the case of a smooth integrator.
\end{abstract}


\maketitle

\section{Basic concepts}

Initially, we introduce the fundamental tools that are required to develop this work. Most concretely, we give some brief notions concerning the Riemann-Liouville fractional integral and classical results about convolutions. During the rest of the document, we will assume that $a,b \in \mathbb{R}$ are real numbers and $\alpha, \beta \in \mathbb{R}^+$ are strictly positive real numbers. It is important to take into account that we will be dealing most of the time with the Banach space $L^p[a,b]$ for $p \in [1,\infty]$ and, hence, all the functional identities that we describe will hold in the whole interval $[a,b]$ except for, at most, a zero Lebesgue measure set. We will use the notation $\mathcal{B}\left(L^p[a,b]\right)$ for the space of linear continuous operators from $L^p[a,b]$ to itself. Finally, we highlight that all the deductions are obviously still valid if we replace $L^p[a,b]$ by $\mathcal{C}([a,b])$.

\begin{Definition} \label{Lio} Given $\alpha \in \mathbb{R}^+$ and $a \in \mathbb{R}$ we define the Riemann-Liouville fractional integral of order $\alpha$ and base point $a$ of a function $f \in L^p[a,b]$ as $$\left(I_{a^+}^{\alpha}f\right)(t):=\int_a^t \frac{(t-s)^{\alpha-1}}{\Gamma(\alpha)}\,f(s) \, ds = \int_a^t \frac{(s-a)^{\alpha-1}}{\Gamma(\alpha)}\,f(t-s+a) \, ds.$$
\end{Definition}

The exact expression for $I_{a^+}^{\alpha}$ is not extremely relevant for our purposes. Of course, we observe that for $\alpha=1$ we recover the classical integral operator with base point $a$. The important fact that we need to take into account is the following proposition, which is a compilation of well-known results that can be found in \cite{Samko}.

\begin{Proposition} \label{Prop} The fractional integral operator $I_{a^+}^{\alpha}$ satisfies several properties:
\begin{itemize}
\item The map $I_{a^+}^{\alpha}: L^p[a,b] \longrightarrow L^p[a,b]$ is well-defined.
\item $I_{a^+}^1$ is the usual integral operator with base point $a$.
\item We have the Index Law $I_{a^+}^{\alpha+\beta} = I_{a^+}^{\alpha} \circ I_{a^+}^{\beta}$ for any $\alpha, \beta>0$.
\item The map $\mathbb{R}^+ \to \mathcal{B}\left(L^p[a,b]\right)$ given by $\alpha \to I_{a^+}^{\alpha}$ is continuous.
\end{itemize}
\end{Proposition}

We are also interested in the ``Riemann-Liouville fractional integral with respect to a function $h$'', which is also defined at \cite{Samko}. Analogously to what is done at \cite{Samko}, in the rest of the document we will assume that $h \in \mathcal{C}^1[a,b]$ with $h'(t) \neq 0$ for any $t \in [a,b]$. Besides, to consider only the case of a left base point of integration, we make the assumption $h'(t)>0$. In other case $h'(t)<0$, and the same result can be obtained after considering a fractional integral with right base point.

\begin{Definition} Given $\alpha \in \mathbb{R}^+$ and $a \in \mathbb{R}$ we define the Riemann-Liouville fractional integral of order $\alpha$ and base point $a$ of a function $f \in L^p[a,b]$ with respect to $h$ as $$\left(I_{h, \, a^+}^{\alpha}f\right)(t):=\int_a^t \frac{(h(t)-h(s))^{\alpha-1}}{\Gamma(\alpha)}\,f(s) \, h'(s)\, ds.$$
\end{Definition}

There are several properties of this operator that are derived in \cite{Samko}. We will only highlight that, for $\alpha =1$, the operator $I_{h, \, a^+}^{1}$ is the Stieltjes integral operator with integrator $h$ and that $I_{h, \, a^+}^{\alpha} \in \mathcal{B}\left(L^p[a,b]\right)$.

With respect to convolutions, we recall that for each function $g \in L^1[a,b]$ we can induce, in a continuous way, a convolution operator after defining a map $C_a:L^1[a,b] \longrightarrow \mathcal{B}(L^p[a,b])$. More specifically, for each $g \in L^1[a,b]$ its associated convolution operator $C_a(g):L^p[a,b] \longrightarrow L^p[a,b]$ is defined as $$(C_a(g)\,f)(t):= (g * f)(t):=\int_{a}^t g(t-s+a) \cdot f(s) \, ds.$$ Under the previous notation, we say that $g$ is the kernel of the convolution operator $C_a(g)$. There are some well-known properties about convolution operators that we summarize below.

Essentially, the convolution operation $*$ is commutative and associative. Moreover, the linear operator $C_a$ is continuous and well-defined, in the sense that $C_a(g)$ is a bounded endomorphism in $L^p[a,b]$ for any $g \in L^1[a,b]$ that changes continuously with respect to the operator norm when moving $g$. Moreover, we will use the following Theorem, due to Titchmarsh in \cite{Tit}, that roughly states that convolution operators are injective, provided that the kernel is different from the zero function on any right neighbourhood of the base point $a$.

\begin{Theorem}[Titchmarsh]
Suppose that $f,g \in L^1[a,b]$ are such that $f*g \equiv 0$. Then, there exist $\lambda,\mu \in \mathbb{R}^+$ such that the following three conditions hold:
\begin{itemize}
\item $f \equiv 0$ (almost everywhere) in the interval $[a,a+\lambda]$,
\item $g \equiv 0$ (almost everywhere) in the interval $[a,a+\mu]$,
\item $\lambda+\mu \geq b-a$.
\end{itemize}
\end{Theorem}

As we already mentioned, the previous result implies that $C_a(g)$ is injective, provided that $g \in L^1[a,b]$ is a non-identically null function at any right neighbourhood of $a$.

In our case, the interest of convolutions and their connection with Riemann-Liouville fractional integral comes from the identity $C_a \left(g_{a, \,\alpha}\right)=I_{a^+}^{\alpha}$ for $\alpha>0$, where we have defined
\begin{equation} \label{Eq}
g_{a,\, \alpha}(t):=\frac{(t-a)^{\alpha-1}}{\Gamma(\alpha)} \in L^1[a,b].
\end{equation}
Thus, many typical tools concerning convolution theory can be applied to study the so-called ``fractional operators''. Finally, and due to technical reasons, we will require also the following result.

\begin{Lemma} \label{Lema}
Consider three Banach spaces $X,Y,Z$ and denote by $\mathcal{B}(X,Y)$ the space of continuous linear maps from $X$ to $Y$ with the norm topology. Then, the operator $\textnormal{Comp}: \mathcal{B}(X,Y) \times \mathcal{B}(Y,Z) \longrightarrow \mathcal{B}(X,Z)$ given by $\textnormal{Comp}(g, f)= f \circ g$ is continuous.
\end{Lemma}

\begin{proof}
From the triangle inequality, we deduce $$\Vert \textnormal{Comp}(g_1,f_1)-\textnormal{Comp}(g_2,f_2)\Vert \leq \Vert f_1 \circ g_1 - f_1 \circ g_2 \Vert + \Vert f_1 \circ g_2 - f_2 \circ g_2 \Vert.$$ Moreover, the right hand side can be bounded from above by $$\Vert f_1 \Vert \cdot \Vert g_1 - g_2 \Vert + \Vert f_1 - f_2 \Vert \cdot \Vert g_2 \Vert.$$ Finally, we observe that if $(g_2,f_2)$ tends to $(g_1,f_1)$ the previous bound goes to zero, since $\Vert g_1 - g_2 \Vert$, $\Vert f_1 - f_2 \Vert$ go to zero and $\Vert g_2 \Vert$ tends to $\Vert g_1 \Vert$ due to the continuity of the norm.
\end{proof}

\begin{Remark} \label{Rema} From the previous result it is straightforward to check that, given a function $g \in \mathcal{B}(Y,Z)$, then the map $ \textnormal{Comp}_{(g,\cdot)}: \mathcal{B}(X,Y) \longrightarrow \mathcal{B}(X,Z)$ defined as $ \textnormal{Comp}_{(g,\cdot)}(f)=f \circ g$ is continuous. Analogously, given a function $f \in \mathcal{B}(X,Y)$, then the map $ \textnormal{Comp}_{(\cdot,f)}:  \mathcal{B}(Y,Z) \longrightarrow \mathcal{B}(X,Z)$ defined as $ \textnormal{Comp}_{(\cdot,f)}(g)=f \circ g$ is continuous.
\end{Remark}

\section{Introduction to the state of the art}

In 1978, Donald Cartwright and John McMullen provided a result that characterized axiomatically the Riemann-Liouvile fractional integral. In this sense, they proved that the Riemann-Liouville fractional integral was the unique definition that extended the integral operator $I_{a^+}^1$, assuming some reasonable hypotheses for the extension. Our goal is to provide a similar result for the Stieltjes integral operator with integrator $h$. As one can expect, the unique family that performs this extension will be the ``Riemann-Liouville fractional integral with respect to $h$''.

Now, we state the Cartwright and McMullen result for the case of real functions. Interested readers can find their original work in \cite{CarMc}, which differs from the following statement in the consideration of complex-valued functions. The consideration of real valued functions allows us to omit the third hypothesis in \cite{CarMc}, regarding the positivity of $J_{a^+}^{\alpha}$.

\begin{Theorem}[Cartwright-McMullen, real version] \label{TheoO}
Given a fixed $a \in \mathbb{R}$, there is only one family of operators $\left(J_{a^+}^{\alpha}\right)_{\alpha > 0}$ on $L^p[a,b]$ satisfying the following conditions:
\begin{enumerate}
\item The operator of order $1$ is the usual integral with base point $a$. That is, $J_{a^+}^1=I_{a^+}^1$. (Interpolation property)
\item The Index Law holds. That is, $J_{a^+}^{\alpha} \circ J_{a^+}^{\beta}=J_{a^+}^{\alpha+\beta}$ for all $\alpha,\beta > 0$. (Index Law)
\item The family is continuous with respect to the parameter. That is, the following map $\textnormal{Ind}_a:\mathbb{R}^+ \longrightarrow \mathcal{B} \left(L^p[a,b] \right)$ given by $\textnormal{Ind}_a(\alpha) = J_{a^+}^{\alpha}$ is continuous, where the codomain has the norm topology. (Continuity)
\end{enumerate}
This family is precisely given by the Riemann-Liouville fractional integrals and, hence, we have $J_{a^+}^{\alpha}=I_{a^+}^{\alpha}$.
\end{Theorem}

Although we will not reproduce the proof presented in \cite{CarMc}, we sketch the main idea. Essentially, the authors split the proof into two steps.

The first step is to show that the result holds when $J_{a^+}^{\alpha}$ is assumed to be a convolution operator. More concretely, the idea is to show that it is enough to study $J_{a^+}^{1/m}$ for $m \in \mathbb{Z}^+$. The main reason is that the Index Law and the continuity with respect to the index allow us to describe $J_{a^+}^{\alpha}$ for $\alpha >0$, provided that we know $J_{a^+}^{1/m}$. When $J_{a^+}^{1/m}$ is assumed to be a convolution operator, it is possible to use appropriated tools from convolution theory, namely Titchmarsh Theorem, to conclude the uniqueness up to the product with a $m$-th root of unity, that is, \[J_{a^+}^{\frac{1}{m}}=e^{\frac{2\pi k i}{m}}I_{a^+}^{\frac{1}{m}}, \textnormal{ where } k \in \{0,1,2,\dots,m-1\}.\] The special treatment for the real case uses that the only possible $m$-roots of unity are $1$ or $-1$, so the Index Law implies \[J_{a^+}^{\frac{p}{q}}=\pm I_{a^+}^{\frac{p}{q}}, \textnormal{ where } p/q \in \mathbb{Q},\] where the ``$-$'' sign could depend, in principle, of each $p/q$. However, only the ``$+$'' sign is possible for each $p/q$, due to the Index Law applied to the composition $J_{a^+}^{p/q}=J_{a^+}^{(p/2q)} \circ J_{a^+}^{(p/2q)}.$ Finally, continuity forces to $J_{a^+}^{\alpha}=I_{a^+}^{\alpha}$ for any $\alpha>0.$

The second step is to show that the result holds when $J_{a^+}^{\alpha}$ is not necessarily a convolution operator. At first, one uses the Index Law and the interpolation property in order to prove that $J_{a^+}^{1/m}$ commutes with any convolution operator with polynomial kernel. Indeed, the continuity of $C_a$ and the density of polynomials in $L^p[a,b]$ imply that it commutes with any convolution operator. The previous fact is crucial to deduce that $I_{a^+}^1 \circ J_{a^+}^{1/m}$ is always a convolution operator, although $J_{a^+}^{1/m}$ was not. Finally, one mimics the discussion developed during the first step, now for the convolution operator $I_{a^+}^1 \circ J_{a^+}^{1/m}$, to conclude the uniqueness for $I_{a^+}^1 \circ J_{a^+}^{1/m}$ up to the product with a $m$-root of unity. Thus, we arrive to a similar situation to the previous one \[I_{a^+}^1 \circ J_{a^+}^{\frac{1}{m}}=e^{\frac{2\pi k i}{m}}I_{a^+}^{1+\frac{1}{m}}, \textnormal{ where } k \in \{0,1,2,\dots,m-1\}.\] Finally, the injectivity of $I_{a^+}^1$ implies \[J_{a^+}^{\frac{1}{m}}=e^{\frac{2\pi k i}{m}}I_{a^+}^{\frac{1}{m}}, \textnormal{ where } k \in \{0,1,2,\dots,m-1\},\] and we can end the argument as in the previous step.

\section{The Stieltjes case}

It is a reasonable question if we can give a similar result to the previous one for the case of the Stieltjes integral operator. We would want to ensure that there is only one continuous interpolation for it such that the Index Law holds. The answer is positive when the integrator is given by a function $h \in \mathcal{C}^1[a,b]$ such that $h'(t) > 0$ for any $t \in [a,b]$. Furthermore, we can give an explicit construction of the interpolation in this case, which is the ``Riemann-Liouville fractional integral with respect to the function $h$''. Instead of developing a technical proof for the result, it will be deduced as a corollary of the Cartwright-McMullen Theorem after suitable remarks.

\begin{Theorem} \label{Theo}
There is only one family of operators $(J_{h,\, a^+}^{\alpha})_{\alpha > 0}$ on $L^p[a,b]$ satisfying the following conditions:
\begin{enumerate}
\item The operator of order $1$ is the usual integral. That is, $J_{h, \, a^+}^1=I_{h, \, a^+}^1$. (Interpolation property)
\item The Index Law holds. That is, $J_{h, \, a^+}^{\alpha} \circ J_{h, \, a^+}^{\beta}=J_{h, \, a^+}^{\alpha+\beta}$ for all $\alpha,\beta > 0$. (Index Law)
\item The family is continuous with respect to the parameter. That is, the following map $\textnormal{Ind}_a:\mathbb{R}^+ \longrightarrow \mathcal{B} \left(L^p[a,b]\right)$ given by $\textnormal{Ind}_{h, \, a}(\alpha) = J_{h, \, a^+}^{\alpha}$ is continuous, where the codomain has the norm topology. (Continuity Property)
\end{enumerate}
The family is precisely given by the ``Riemann-Liouville fractional integral with respect to the function $h$''.
\end{Theorem}

\begin{proof}
Consider the operator $R_{h}: L^p[h(a),h(b)] \longrightarrow L^p[a,b]$ given by $R_{h}(f)=f \circ h$. Since $h$ is continuously differentiable and $h'(t) > 0$ when $t \in [a,b]$, it is a consequence of the Change of Variables Theorem that $R_h$ is well-defined, meaning that $f \circ h \in L^p[a,b]$ when $f \in L^p[h(a),h(b)]$. Although $h$ is not necessarily linear, it is straightforward to check that $R_h$ is an invertible linear operator, where $R_h^{-1}=R_{h^{-1}}$. To see that the operator $R_h$ is continuous, we recall that, for a function $f \in L^p[h(a),h(b)]$, we have $$\Vert f \Vert_{L^p[h(a),h(b)]} = \int_{h(a)}^{h(b)} \left \vert f(t)\right \vert \, dt= \int_a^b \left \vert f(h(t))\right \vert \cdot \left \vert h'(t) \right \vert \, dt \geq m \cdot \int_a^b \left \vert f(h(t))\right \vert \, dt,$$ where $m = \min\{ \vert h'(t) \vert \in \mathbb{R}^+: t \in [a,b] \} >0$ exists since $\vert h' \vert$ is continuous on the compact interval $[a,b]$ and does not vanish. Thus, $$\Vert R_h(f) \Vert_{L^p[a,b]}= \int_a^b \vert f(h(t)) \vert \, ds \leq \frac{1}{m} \, \Vert f \Vert_{L^p[h(a),h(b)]}$$ and we have proved that $R_h$ is continuous.

The previous properties concerning $R_h$ are of our interest because
\begin{equation} \label{Eq4}
I_{h,\,a^+}^1 = R_h \circ I_{h(a)^+}^1 \circ R^{-1}_h.
\end{equation} This claim follows by direct calculation, since
\begin{equation*}
\left(R_h \circ I_{h(a)^+}^1 \circ R^{-1}_h  f\right)(t)=\left(R_h \circ I_{h(a)^+}^1 (f \circ h^{-1}) \right)(t)=R_h \left(\int_{h(a)}^t f\left(h^{-1}(s)\right) \, ds \right)
\end{equation*}
and the application of $R_h$ and the Change of Variables Theorem allow to rewrite the previous right hand side as
\begin{equation*}
\int_{h(a)}^{h(t)} f\left(h^{-1}(s)\right) \, ds=\int_{a}^{t} f(s) \cdot h'(s) \, ds = \left(I_{h,\,a^+}^1 f\right)(t).
\end{equation*}

In fact, it is immediate to show by iterated composition that $I_{h,\,a^+}^n = R_h \circ I_{h(a)^+}^n \circ R^{-1}_h$ for any positive integer $n$. Now, our intuition tells us that if $J_{h(a)^+}^{\alpha}$ is a ``nice interpolation'' for $I_{h(a)^+}^{\alpha}$, then the definition $R_h \circ J_{h(a)^+}^{\alpha} \circ R^{-1}_h$ should be a ``nice interpolation'' for $I_{h,\,a^+}^{\alpha}.$ Conversely, a choice $J_{h, \, a^+}^{\alpha}$ fulfilling the hypotheses in Theorem \ref{Theo}, should imply that
\begin{equation} \label{Eq3}
K_{h(a)^+}^{\alpha}:=R_h^{-1} \circ J_{h, \, a^+}^{\alpha} \circ R_h.
\end{equation}
is under the hypotheses of the Cartwright-McMullen Theorem \ref{TheoO}. This intuition can be confirmed after the following three remarks. Before doing this, note that the continuity of $J_{h, \, a^+}^{\alpha}$, $R_h$ and $R_h^{-1}$ imply that $K_{h(a)^+}^{\alpha} \in \mathcal{B}\left(L^p[h(a),h(b)]\right)$. 
\begin{enumerate}
\item If $J_{h, \, a^+}^{\alpha}=I_{h, \, a^+}^{\alpha}$, then $K_{h(a)^+}^{1}=I_{h(a)^+}^1$. This is a consequence of Equations (\ref{Eq4}), (\ref{Eq3}), and the Interpolation Property (first hypothesis in Theorem \ref{Theo}).
\item If $J_{h, \, a^+}^{\alpha} \circ J_{h, \, a^+}^{\beta} = J_{h, \, a^+}^{\alpha+\beta}$, then $K_{h(a)^+}^{\alpha} \circ K_{h(a)^+}^{\beta} = K_{h(a)^+}^{\alpha+\beta}$. This is a consequence of Equation (\ref{Eq3}) and the Index Law (second hypothesis in Theorem \ref{Theo}).
\item If the map $\textnormal{Ind}_{h, \, a}(\alpha)=J_{h, \, a^+}^{\alpha}$ is continuous, then $\textnormal{Ind}_{h(a)}(\alpha)=K_{h(a)^+}^{\alpha}$ is continuous. This part is not completely straightforward, since it is a consequence of (\ref{Eq3}) and the Continuity Property (third hypothesis in Theorem \ref{Theo}), but the continuity of the composition operator described in Lemma \ref{Lema} and Remark \ref{Rema} also plays a key role. We just see that we can write $\textnormal{Ind}_{h(a)}(\alpha)$ as a composition of continuous operators \[\textnormal{Ind}_{h(a)}(\alpha)=\textnormal{Comp}_{ \left(\cdot, R_h^{-1} \right)} \circ \textnormal{Comp}_{\left(R_h,\cdot\right)} \circ \textnormal{Ind}_{h, \, a}(\alpha)\]
\end{enumerate}
Therefore, since $R_h$ and $R_h^{-1}$ are bijective, two different choices for $J_{h, \, a^+}^{\alpha}$ would induce two different possibilities for $J_{h(a)^+}^{\alpha}$ in the Cartwright-McMullen Theorem. Since this is not possible, there is also a unique choice for $J_{h, \, a^+}^{\alpha}$ that is induced from the unique choice for $J_{h(a)^+}^{\alpha}$, which is $J_{h(a)^+}^{\alpha}=I_{h(a)^+}^{\alpha}$. Consequently, the unique possibility for $J_{h, \, a^+}^{\alpha}$ will be given by \[I_{h, \, a^+}^{\alpha}= R_h \circ I_{h(a)^+}^{\alpha} \circ R_h^{-1}.\]
\end{proof}



\end{document}